\documentclass[twoside,12pt]{article}

\usepackage{amssymb, amsmath, mathrsfs, amsthm}
\usepackage{graphicx}
\usepackage{color}
\usepackage{tikz}
\usepackage[top=2cm, bottom=2cm, left=1.8cm, right=1.8cm]{geometry}
\usepackage{float, caption, subcaption}
\usepackage{diagbox}
\usepackage{algorithm}
\usepackage{algpseudocode}

\usepackage[colorlinks,
  linkcolor=blue,
  anchorcolor=blue,
  citecolor=blue]{hyperref}

\newcommand{\bd}{\begin{description}}
\newcommand{\ed}{\end{description}}
\newcommand{\bi}{\begin{itemize}}
\newcommand{\ei}{\end{itemize}}
\newcommand{\be}{\begin{enumerate}}
\newcommand{\ee}{\end{enumerate}}
\newcommand{\beq}{\begin{equation}}
\newcommand{\eeq}{\end{equation}}
\newcommand{\beqs}{\begin{eqnarray*}}
\newcommand{\eeqs}{\end{eqnarray*}}
\newcommand{\flr}[1]{\left\lfloor #1 \right\rfloor}
\newcommand{\floor}[1]{\flr{#1}}
\newcommand{\ceil}[1]{\left\lceil #1 \right\rceil}

\definecolor{DarkGreen}{rgb}{0.2, 0.6, 0.3}

\usepackage{multicol}
\usepackage{multirow}
\usepackage{makecell}

\newtheorem{theorem}{Theorem}[section]

\newtheorem{lemma}[theorem]{Lemma}
\newtheorem{definition}[theorem]{Definition}
\newtheorem{corollary}[theorem]{Corollary}

\newtheorem{proposition}[theorem]{Proposition}

\newtheorem{problem}[theorem]{Problem}

\newcommand{\ind}{\mathrm{ind}}
\newcommand{\BRs}{\mathrm{R}^{\sharp}}
\newcommand{\LaT}{\mathrm{La}}
\newcommand{\LT}{\mathrm{L}}

\begin{document}

\title{\textbf{Erd\H{o}s-Gy\'{a}rf\'{a}s problem for partially ordered sets}}

\author{
Gyula O.\,H. Katona\thanks{HUN-REN Alfr\'ed R\'enyi Institute of Mathematics, Reaaltanoda utca 13--15, Budapest 1053, Hungary. E-mail: \texttt{katona.gyula.oh@renyi.hu}}
\and
Yaping Mao\thanks{Corresponding author. Academy of Plateau Science and Sustainability, and School of Mathematics and Statistics, Qinghai Normal University, Xining, Qinghai 810008, China. E-mail: \texttt{yapingmao@outlook.com}; \texttt{myp@qhnu.edu.cn}}
}

\date{}

\maketitle
\begin{abstract}
Given integers $p,q,t$ with $1 \le t \le p$ and $1 \le q \le h_p(t)$, a strong $(p,q,t)$-coloring of the Boolean lattice $B_n$ is a coloring of its $t$-chains such that every induced copy of $B_p$ in $B_n$ uses at least $q$ colors on its $t$-chains. Let $f_t^{\sharp}(n,p,q)$ denote the minimum number of colors in such a coloring. We study this Boolean-lattice analogue of the Erd\H{o}s-Gy\'{a}rf\'{a}s function.
We first show that every finite poset strongly embeds into a Boolean lattice. Combined with a structural Ramsey theorem for finite posets with linear extensions, this implies the existence of the strong Boolean Ramsey number $\mathrm{R}^{\sharp}_{k,t}(\mathcal{B}\mid Q)$ for every integer $k\ge1$, every $t\ge1$, and every nonempty finite poset $Q$. In particular, this gives an affirmative answer to a problem of Cox and Stolee and yields the existence of $f_t^{\sharp}(n,p,2)$.
Next, using the symmetric Lov\'asz local lemma, we obtain a probabilistic upper bound on $f_t^{\sharp}(n,p,q)$. Finally, we prove lower bounds by combining Tur\'an-type extremal estimates for $t$-chains, a double-counting argument, and a generalized Lubell-type framework for $t$-chains.

\noindent\textbf{Keywords:}
Ramsey theory; Boolean lattice; Erd\H{o}s-Gy\'{a}rf\'{a}s function; $t$-chain; Lov\'asz local lemma; Lubell function; strong Boolean Ramsey number.

\medskip
\noindent\textbf{MSC 2020:}
05D05; 05D10; 06A07.
\end{abstract}
\section{Introduction}

In 1975, Erd\H{o}s and Shelah \cite{Er75,Er81} introduced the following graph-coloring refinement of the classical Ramsey problem.

\begin{definition}
Let $p$ and $q$ be integers with $1 \le q \le \binom{p}{2}$. A \emph{$(p,q)$-coloring} of the complete graph $K_n$ is an edge-coloring of $K_n$ in which every copy of $K_p$ receives at least $q$ distinct colors on its edges.
\end{definition}

\begin{definition}
For integers $p$ and $q$ with $1 \le q \le \binom{p}{2}$, let $f(n,p,q)$ denote the minimum number of colors in a $(p,q)$-coloring of $K_n$.
\end{definition}

If $r_k(p)$ is the usual $k$-color Ramsey number, then
\[
f(n,p,2)=k
\text{ if and only if }
r_k(p)>n \text{ and } r_{k-1}(p)\le n.
\]
Indeed, $r_k(p)>n$ is equivalent to the existence of a $k$-edge-coloring of $K_n$ with no monochromatic $K_p$, while $r_{k-1}(p)\le n$ says that $(k-1)$ colors do not suffice.

The graph Erd\H{o}s-Gy\'{a}rf\'{a}s problem has been studied extensively. Mubayi \cite{Mu98} proved that $f(n,4,3)\le e^{O(\sqrt{\log n})}$ by an explicit construction. Eichhorn and Mubayi \cite{EM00} extended this to show the same type of upper bound for $f(n,p,2\ceil{\log p}-2)$ for every $p\ge5$. Erd\H{o}s and Gy\'{a}rf\'{a}s \cite{EG97} determined the threshold values of $q$ for which $f(n,p,q)$ is linear, quadratic, or asymptotic to $\binom{n}{2}$. Further developments can be found in \cite{BEHK23,CFLS152,CH23,EG97,EM00,Mu98,Yip26}.

A \emph{Boolean lattice of dimension $n$}, denoted $B_n=2^{[n]}$, is the power-set lattice of an $n$-element set ordered by inclusion. Poset Ramsey problems in Boolean lattices have attracted significant interest in recent years; see \cite{AW23,GKM,KMOW25,GMT23,BP2021,Winter25,Winter23}.

There is also an important extremal-set-theoretic precursor to the present $t$-chain coloring problem.  Write $C_s$ for the $s$-element chain.  Katona \cite{Katona73} determined the maximum number of $2$-chains in a $C_3$-free family, in our notation the weak parameter $\LaT_2(n,C_3)$.  Gerbner and Patk\'{o}s \cite{GerbnerPatkos08} subsequently determined the corresponding chain case $\LaT_t(n,C_{t+k})$ for all positive integers $t$ and $k$, using $l$-chain profile vectors.  Gerbner, Keszegh, and Patk\'{o}s \cite{GerbnerKeszeghPatkos20} introduced the more general problem of maximizing the number of copies of one poset $Q$ inside a family avoiding another poset $P$, and treated several small cases.  For the chain-copy specialization, Gerbner, Methuku, Nagy, Patk\'{o}s, and Vizer \cite{GerbnerMethukuNagyPatkosVizer19} proved a dichotomy analogous in spirit to graph Tur\'{a}n theory: if the height $h(P)$ of the forbidden poset is larger than $t$, then the order of magnitude of the maximum number of $t$-chains in a $P$-free family is the same as in the chain-forbidden case $C_{t+1}$; if $h(P)\le t$, then the order of magnitude is smaller.  Balogh, Martin, Nagy, and Patk\'{o}s \cite{BaloghMartinNagyPatkos22} further analyzed the height-two case for $t=2$, including an exact result for the butterfly poset.  These works consider weak, non-induced forbidden copies; for forbidden chains the weak and induced notions coincide.  The present paper develops a complementary Ramsey-coloring viewpoint, with special emphasis on strong, induced Boolean-lattice copies.

It is natural to formulate the Erd\H{o}s-Gy\'{a}rf\'{a}s problem for Boolean lattices.

\begin{definition}\label{def:pqtcolor}
Let $p,q,t$ be integers with $1 \le t \le p$ and $1 \le q \le h_p(t)$, where $h_p(t)$ is the number of $t$-chains in $B_p$. A \emph{$(p,q,t)$-coloring} of $B_n$ is a coloring of the $t$-chains of $B_n$ such that every copy of $B_p$ in $B_n$ uses at least $q$ colors on its $t$-chains. A \emph{strong $(p,q,t)$-coloring} is defined analogously with ``copy'' replaced by ``induced copy''.
\end{definition}

\begin{definition}\label{def:ftnpq}
Let $p,q,t$ be as above. We write $f_t(n,p,q)$ for the minimum number of colors in a $(p,q,t)$-coloring of $B_n$, and $f_t^\sharp(n,p,q)$ for the minimum number of colors in a strong $(p,q,t)$-coloring of $B_n$.
\end{definition}

For $q=2$, the strong coloring function is equivalent to the strong Boolean Ramsey number for $t$-chains. Let
$\mathcal B=\{B_n:n\ge1\}$.
Then
\begin{equation}\label{eq:ft-Ramsey-equivalence}
f_t^\sharp(n,p,2)=k
\text{ if and only if }
\BRs_{k,t}(\mathcal B\mid B_p)>n
\text{ and }
\BRs_{k-1,t}(\mathcal B\mid B_p)\le n.
\end{equation}

In this paper we study the strong Boolean-lattice Erd\H{o}s-Gy\'{a}rf\'{a}s function and its Ramsey-theoretic consequences. Our results are organized as follows.

In Section~\ref{sec_2-3}, we establish several basic properties of the Erd\H{o}s-Gy\'{a}rf\'{a}s functions for posets, including monotonicity relations and comparisons between the weak and strong versions.

We then turn to the existence problem for strong Boolean Ramsey numbers. We show that every finite poset strongly embeds into a Boolean lattice. Combined with a structural Ramsey theorem for finite posets with linear extensions, this yields the existence of the strong Boolean Ramsey number
$\mathrm{R}^{\sharp}_{k,t}(\mathcal{B}\mid Q)$
for every integer $k\ge1$, every $t\ge1$, and every nonempty finite poset $Q$; see Section~\ref{sec6}. In particular, this gives an affirmative answer to a problem of Cox and Stolee and implies the existence of $f_t^{\sharp}(n,p,2)$.

In Section~\ref{sec3}, using the symmetric Lov\'asz local lemma (Theorem~\ref{th-LLL-SV}), we prove the following upper bound.

\begin{theorem}\label{th-upper-Pr}
Let \(p, q, t\) be positive integers with $1 \le t \le p$ and $1 \le q \le h_p(t)$, where \(h_p(t)\) is defined in Proposition~\ref{th-La2}. For all integers \(n \ge p\), there exists a constant \(c=c(p,q,t)\) such that
\[
f_t^{\sharp}(n,p,q)\le
c\cdot 2^{\frac{(n-p)\binom{p}{\lfloor p/2\rfloor}(1+o(1))}{h_p(t)-q+1}}.
\]
\end{theorem}

In Section~\ref{sec4}, we derive lower bounds for the strong Boolean-lattice Erd\H{o}s-Gy\'{a}rf\'{a}s function by combining Tur\'an-type extremal estimates for $t$-chains, a double-counting argument, and a generalized Lubell-type framework. Our main result is the following.

\begin{theorem}\label{th123}
Let $n\ge 2p+1$, let $1\le t\le p$, and let $2\le q\le h_p(t)$. Then
\[
f_t^\sharp(n,p,q)
\ge
\max\left\{
\ceil{\frac{h_n(t)}{h_n(t)-\floor{\frac{n}{p+1}}}},
\ceil{\frac{h_n(t)\,C_{\min}(p,n,t;\mathcal T_n(t))}{g(p,n)\,h_p(t)}}
\right\}.
\]
\end{theorem}

Here $h_n(t)$ denotes the number of $t$-chains in $B_n$ (Proposition~\ref{th-La2}), $g(m,N)$ denotes the number of induced copies of $B_m$ in $B_N$, and Theorem~\ref{number_b_n} provides an upper bound on $g(m,N)$. Moreover, \(C_{\min}(p,n,t;\mathcal T)\) denotes the minimum number of induced copies of \(B_p\) in \(B_n\) containing a given \(t\)-chain in \(\mathcal T\subseteq \mathcal T_n(t)\) (Lemma~\ref{lem-corrected-tchain-concentration}). For some parameter choices, this bound is attained.

Finally, in Section~\ref{sec5}, we apply the upper bound on $f_t^{\sharp}(n,p,2)$ to obtain new lower bounds for strong Boolean Ramsey numbers. In particular, we show that our bound in Corollary~\ref{cor-R-lower} strictly improves, for sufficiently large $k$, the uniform lower bound of Katona et al.~\cite{KMOW25} stated in Corollary~\ref{cor-lower-Boolean-uniform}; see Proposition~\ref{pro-compa}.

\section{Preliminaries}

A \emph{partially ordered set} (abbreviated \emph{poset}) is a pair $(P, \leq)$, where $P$ is a non-empty set (the \emph{ground set}) and $\leq$ is a binary relation on $P$ satisfying three axioms: reflexivity ($x \leq x$ for all $x \in P$), anti-symmetry (if $x \leq y$ and $y \leq x$ for $x, y \in P$, then $x = y$), and transitivity (if $x \leq y$ and $y \leq z$ for $x, y, z \in P$, then $x \leq z$).
When the order relation $\leq$ is clear from context, we denote the poset simply by $P$.
Two distinct elements $x, y \in P$ are \emph{comparable} if $x \leq y$ or $y \leq x$; otherwise, they are \emph{incomparable}. A \emph{$t$-chain} in $P$ is a totally ordered subset of $P$ of size $t$ (equivalently, a set of $t$ distinct elements where every pair is comparable).
The \emph{$n$-antichain}, denoted $A_n$, is the poset on $n$ elements where every pair of distinct elements is incomparable.

For a poset $P$, $|P|$ denotes the cardinality of $P$;
The \emph{height} $h(P)$ of $P$ is the maximum size of a chain in $P$.
For posets $P$ and $Q$, an injection $f: P \to Q$ is a \emph{weak embedding} if $f(x) \leq f(y)$ in $Q$ whenever $x \leq y$ in $P$; in this case, $f(P)$ is a \emph{copy} of $P$ in $Q$; An injection $f: P \to Q$ is a \emph{strong embedding} if $f(x) \leq f(y)$ in $Q$ if and only if $x \leq y$ in $P$; in this case, $f(P)$ is an \emph{induced copy} of $P$ in $Q$.

A family $\mathcal F\subseteq B_n$ is always viewed as the induced subposet of $B_n$ under set inclusion. For any poset $R$, let
$A_t(R)=\{\text{all }t\text{-chains of }R\}$ and $h(R,t)=|A_t(R)|$.
In particular, we write
$\mathcal T_n(t)=A_t(B_n)$ and $h_n(t)=h(B_n,t).$
Finally, $g(p,n)$ denotes the number of induced copies of $B_p$ in $B_n$.

A \emph{$k$-coloring of the $t$-chains} of a poset is a function assigning one of $k$ colors (typically labeled $1, 2, \dots, k$) to each $t$-chain of the poset. For $t = 1$, this reduces to a $k$-coloring of the elements of the poset.
A colored poset $P$ is \emph{monochromatic} with respect to a $t$-chain coloring if all $t$-chains of $P$ are assigned the same color.

Katona et al. \cite{KMOW25} determined the exact number for $h_n(t)$ as follows.
\begin{proposition}{\upshape \cite{KMOW25}}\label{th-La2}
For positive integers \(p \geq t\), the number of \(t\)-chains in \(B_p\) is
\begin{equation*}
h_p(t)=\sum_{j=0}^{t-1}(-1)^j\binom{t-1}{j}(t+1-j)^p = \Theta((t + 1)^p).
\end{equation*}
\end{proposition}

Axenovich and Walzer \cite{AW17} gave the asymptotic behavior of the number of strong embeddings as
$$
\frac{N !}{(N-m) !}(a(m)-m)^{N-m} \leq g(m, N) \leq \frac{N !}{(N-m) !} a(m)^{N-m}.
$$
Since it is known that $a(m)=2^{{m\choose \left\lfloor m/2\right\rfloor}(1+O(\log_2 m/m))}$
(see \cite{KM}) and $\frac{N !}{(N-m) !}\leq N^m=2^{m\log_2 N}$,
we obtain the following useful bound.
\begin{theorem}[\cite{AW17}, Count of Strong Embeddings]\label{number_b_n}
Let $m,N$ be two integers with $m \leq N$.
Then
$$
g(m, N) = 2^{(N-m){m\choose \left\lfloor m/2\right\rfloor}(1+o(1))}
<
2^{2{m\choose \left\lfloor m/2\right\rfloor}(N-m)}.
$$
\end{theorem}
\subsection{Ramsey theory on posets}
Ramsey theory is a branch of combinatorics that investigates the conditions under which a combinatorial object necessarily contains given smaller substructures \cite{Ramsey30}. It has foundational applications across diverse mathematical and computational disciplines, including algebra, geometry, mathematical logic, ergodic theory, theoretical computer science, information theory, number theory, and set theory. Comprehensive surveys of the field can be found in \cite{GRS90, Rosta04}.

Ramsey theory on posets was initiated by Ne\v{s}et\v{r}il and R\"{o}dl \cite{NesetrilRodl}, who focused on induced copies of posets. Kierstead and Trotter \cite{KiersteadTrotter} extended this line of research by considering general posets as host structures (instead of restricting to Boolean lattices) and investigating Ramsey-type problems via key poset parameters: cardinality, height, and width. For a comprehensive overview of related work, we refer the reader to \cite{DKT91, KiersteadTrotter, McColm, TrotterRamsey}. Recently, several authors \cite{AW17, CS18, GundersonRodlSidorenko, JohnstonLuMilans, LT22} have focused on identifying induced copies of posets within Boolean lattices.

\begin{definition}\label{def-BRn}
Let $\mathcal{B} = \{B_n : n \geq 1\}$ denote the family of Boolean lattices. The \textit{Boolean Ramsey number} $\operatorname{R}(\mathcal{B}\,|\,P, Q)$ is the smallest $n$ such that any red/blue-coloring of $B_n$ contains either a red copy of $P$ or a blue copy of $Q$.
\end{definition}

The Ramsey number $\operatorname{R}(\mathcal{B}\,|\,B_s, B_t)$ has been extensively studied \cite{AW17, LT22, CS18, GMT23, BP2021, Walzer15}. Axenovich and Winter \cite{AW23} studied $\operatorname{R}(\mathcal{B}\,|\, P, B_t)$ for a fixed poset $P$ and $t$-dimensional Boolean lattice $B_t$ as $t$ grows. Winter investigated Boolean Ramsey numbers of Boolean lattices versus complete multipartite posets, parallel chain compositions, and $N$-shaped posets \cite{Winter23, WinterIII, WinterII}, and Erdős-Hajnal problems for posets \cite{Winter25}.

Rainbow generalizations of these Ramsey numbers are studied in \cite{CGLMNPV22, CCLL20, Patkos20}.

Griggs, Stahl, and Trotter \cite{GST84} first investigated $t$-chain structures in 1984, and their generalized Lubell inequality provides the foundation for our $t$-chain Lubell function framework.
Cox and Stolee \cite{CS18} introduced a definition of the poset Ramsey number framed in the context of pographs and weak embeddings, with Boolean lattices as the host posets under consideration. Katona et al. \cite{KMOW25} introduced a definition formulated for general posets and accommodates both strong and weak embeddings.

\begin{definition}\label{Defi-Wposet}
For a family $\mathcal{Q}=\{Q_{n} : n\geq 1\}$ of host posets
such that $Q_n \subseteq Q_{n+1}$ and $|Q_n|<|Q_{n+1}|$ for each $n$,
and given $k$ posets $P_1,P_2,\ldots,P_k$,
the \emph{weak poset Ramsey number
for $t$-chains},
denoted by $\operatorname{R}_{k,t}(\mathcal{Q}\,|\,P_1,P_2,\ldots,P_k)$,
is the smallest number $N$
such that for any $k$-coloring of $t$-chains in $Q_N \in \mathcal{Q}$,
there is a monochromatic copy of $P_i$ in color $i$ for some $1\leq i\leq k$.
\end{definition}

We simply write $\operatorname{R}_{k,t}(\mathcal{Q}\,|\,P)$ for $\operatorname{R}_{k,t}(\mathcal{Q}\,|\,P_1,P_2,\ldots,P_k)$ if $P_1=\cdots=P_k=P$.
If $t=1$, then we write $\operatorname{R}_{k}(\mathcal{Q}\,|\,P_1,P_2,\ldots,P_k)$. Furthermore, if $t=1$ and $k=2$, then we write $\operatorname{R}(\mathcal{Q}\,|\,P_1,P_2)$.

We analogously define the \emph{strong poset Ramsey number} $\operatorname{R}_{k,t}^{\sharp}(\mathcal{Q}\,|\,P_1,P_2,\ldots,P_k)$ for induced copy of $P_i$.
In addition,
we use similar abbreviations such as
$\operatorname{R}_{k,t}^{\sharp}(\mathcal{Q}\,|\,P)$,
$\operatorname{R}_{k}^{\sharp}(\mathcal{Q}\,|\,P_1, P_2, \dots , P_k)$,
and
$\operatorname{R}^{\sharp}(\mathcal{Q}\,|\,P_1, P_2)$.
Recall that $\mathcal{B}=\{B_{n} : n\geq 1\}$ denotes
the family of Boolean lattices.
Then
$\operatorname{R}_{k,t}(\mathcal{B}\,|\,P_1,P_2,\ldots,P_k)$
is called the \emph{Boolean Ramsey number}, see \cite{CS18}.
Similarly,
the variation of strong embeddings has been considered:
in \cite{AW17, LT22},
$\operatorname{R}_{k,t}^{\sharp}(\mathcal{B}\,|\,P_1,P_2,\ldots,P_k)$
is called simply
the \emph{strong Boolean Ramsey number}. We simply write $\operatorname{R}_{k,t}(P)$ for $\operatorname{R}_{k,t}(\mathcal{B}\,|\,P_1,P_2,\ldots,P_k)$ if $P_1=\cdots=P_k=P$.
Let $\mathcal{C}=\{C_{n} : n\geq 1\}$ be a
family of chains,
where $C_n$ denotes the $n$-chain.
Then
$\operatorname{R}_{k,t}(\mathcal{C}\,|\,P_1,P_2,\ldots,P_k)$
is called \emph{the chain Ramsey number},
see
\cite{CS18}.

As mentioned in \cite{KMOW25},
these concepts are extension of some known poset Ramsey problems, like
Boolean Ramsey number (Definition \ref{def-BRn}), chain Ramsey number (see
\cite{CS18}), and ordered Ramsey number for hypergraphs (see \cite{BJV19, CFLS17}).

Cox and Stolee \cite[Section 3]{CS18}
studied the case of $k$-coloring of $t$-chains in a Boolean lattice
with $k \geq 3$ and $t \geq 2$,
mainly for particular posets with $t = 2$. Katona et al. \cite{KMOW25} gave several lower and upper bounds on
the weak and strong poset Ramsey number for $t$-chains.

Cox and Stolee \cite[Section 6]{CS18} also posed the following problem.
\begin{problem}\label{prob1}
Does $\mathrm{R}_{k,t}^{\sharp}(\mathcal{B}\,|\,Q)$ exist for a poset $Q$?
\end{problem}

For $t=1$, this question has received attention; however, for $t\geq 2$, they mentioned that it is not even obvious that such an $N$ exists.

Katona et al. \cite{KMOW25} obtained the following lower bound for $\operatorname{R}_{k,t}^{\sharp}(\mathcal{B}\,|\,B_{m_1},B_{m_2},\ldots,B_{m_k})$.
\begin{theorem}{\upshape \cite{KMOW25}}\label{th-lower-Boolean}
Let $k,t, m_1, m_2,\ldots,m_k$ be integers with $k\geq 3$,
and $1 \leq t-1 \leq m_1\leq m_2 \leq \cdots \leq m_k$. Then
\begin{equation*}
\begin{split}
&\operatorname{R}_{k,t}^{\sharp}(\mathcal{B}\,|\,B_{m_1},B_{m_2},\ldots,B_{m_k})
>
\min\left\{
m_1 + \frac{h_{m_1}(t) +(h_{m_1}(t)-1)\log_2(k-1) -1}{4{m_{1}\choose\left\lfloor m_{1}/2\right\rfloor}}, \ 
m_k + \frac{h_{m_k}(t) -1}{2{m_{k}\choose\left\lfloor m_{k}/2\right\rfloor}} \right\}.
\end{split}
\end{equation*}
\end{theorem}

The following corollary is immediate.
\begin{corollary}\label{cor-lower-Boolean-uniform}
Let $k,t, p$ be integers with $k\geq 3$ and $1 \leq t-1 \leq p$. Then
\begin{equation*}
\operatorname{R}_{k,t}^{\sharp}(B_p)
> \min\left\{
p + \frac{h_{p}(t) +(h_{p}(t)-1)\log_2(k-1) -1}{4\binom{p}{\left\lfloor p/2\right\rfloor}}, \ 
p + \frac{h_{p}(t) -1}{2\binom{p}{\left\lfloor p/2\right\rfloor}} \right\},
\end{equation*}
where $h_p(t)$ denotes the number of $t$-chains in $B_p$ (defined in Proposition \ref{th-La2}).
\end{corollary}
\subsection{Forbidden subposet problems}
Katona and Tarján \cite{KT83} introduced the problem of determining
$\operatorname{La}(n, P)$, which is the maximum size of a family $\mathcal{F} \subseteq 2^{[n]}$ with no copy of $P$ (weak embedding).
Let
$\operatorname{La}^{\sharp}(n, P)$ be the maximum size of a family $\mathcal{F} \subseteq 2^{[n]}$ with no induced copy of $P$ (strong embedding).
For more details on this topic, we refer to a survey paper \cite{GL15} and the book \cite{GP19}.

We present a natural generalization of these classical extremal poset problems, shifting the count from family sizes to the maximum number of $t$-chains in $P$-free families. This generalization is key to deriving lower bounds of Erd\H{o}s-Gy\'{a}rf\'{a}s function for Boolean lattices.

We shall use the following $t$-chain analogues.

\begin{definition}\label{def:La_t}
Let $P$ be a finite poset and let $\mathcal F\subseteq B_n$. We write $A_t(\mathcal F)$ for the set of all $t$-chains in $\mathcal F$. Define
\[
\LaT_t(n,P)=\max\{|A_t(\mathcal F)|:\mathcal F\subseteq B_n\text{ contains no weak copy of }P\}
\]
and
\[
\LaT_t^\sharp(n,P)=\max\{|A_t(\mathcal F)|:\mathcal F\subseteq B_n\text{ contains no induced copy of }P\}.
\]
\end{definition}

\begin{definition}\label{def:L_t}
Let $\mathcal C_t\subseteq \mathcal T_n(t)$. We say that $\mathcal C_t$ is \emph{weakly $(t,P)$-free} if there is no weak copy $S\subseteq B_n$ of $P$ with $A_t(S)\subseteq \mathcal C_t$, and \emph{strongly $(t,P)$-free} if there is no induced copy $S\subseteq B_n$ of $P$ with $A_t(S)\subseteq \mathcal C_t$. Define
\[
\LT_t(n,P)=\max\{|\mathcal C_t|:\mathcal C_t\subseteq \mathcal T_n(t)\text{ is weakly $(t,P)$-free}\}
\]
and
\[
\LT_t^\sharp(n,P)=\max\{|\mathcal C_t|:\mathcal C_t\subseteq \mathcal T_n(t)\text{ is strongly $(t,P)$-free}\}.
\]
\end{definition}

\begin{proposition}\label{prop:weak-strong-monotone}
For every finite poset $P$ and every integer $t\ge1$,
\[
\LaT_t^\sharp(n,P)\ge \LaT_t(n,P)
\qquad\text{and}\qquad
\LT_t^\sharp(n,P)\ge \LT_t(n,P).
\]
\end{proposition}

\begin{proof}
Every strong embedding is, in particular, a weak embedding. Hence every weakly $P$-free family is also strongly $P$-free, and every weakly $(t,P)$-free set of $t$-chains is also strongly $(t,P)$-free.
\end{proof}

\begin{proposition}\label{prop:La-vs-L}
Let $P$ be a finite poset and let $t\ge1$. Assume that every element of $P$ belongs to at least one $t$-chain of $P$. Then
\[
\LaT_t(n,P)\le \LT_t(n,P)
\qquad\text{and}\qquad
\LaT_t^\sharp(n,P)\le \LT_t^\sharp(n,P).
\]
\end{proposition}

\begin{proof}
We prove the weak statement; the strong one is identical.

Let $\mathcal F\subseteq B_n$ contain no weak copy of $P$, and put $\mathcal C_t=A_t(\mathcal F)\subseteq \mathcal T_n(t)$. Suppose, for contradiction, that $\mathcal C_t$ is not weakly $(t,P)$-free. Then there exists a weak copy $S\subseteq B_n$ of $P$ such that $A_t(S)\subseteq \mathcal C_t$. By assumption, every element of $S$ belongs to some $t$-chain of $S$, and every such $t$-chain lies in $\mathcal C_t=A_t(\mathcal F)$. Therefore every element of $S$ belongs to $\mathcal F$, contradicting that $\mathcal F$ is weakly $P$-free. Hence $\mathcal C_t$ is weakly $(t,P)$-free, and so $|A_t(\mathcal F)|\le \LT_t(n,P)$.

Taking the maximum over all admissible $\mathcal F$ yields $\LaT_t(n,P)\le \LT_t(n,P)$.
\end{proof}

\subsection{Relations on the Erd\H{o}s-Gy\'{a}rf\'{a}s function}
\label{sec_2-3}

Based on the definitions of weak/strong embeddings and the corresponding coloring functions $f_t(n, p, q)$ and $f_t^{\sharp}(n, p, q)$, we first establish a fundamental relation between these two minimum color-count functions as follows.
\begin{proposition}\label{prop-ft-ftsharp-ineq}
For all positive integers $p, q, t$ with $1 \leq t < p+1$ and $1 \leq q \leq h_p(t)$,
$f_t(n, p, q) \geq f_t^{\sharp}(n, p, q)$.
\end{proposition}

\begin{proof}
Let $\mathcal{C}(B_p,B_n)$ (respectively, $\mathcal{C}^{\sharp}(B_p,B_n)$) be the set of all copies (respectively, induced copies) of $B_p$ in $B_n$.
Since every strong embedding is, in particular, a weak embedding, we have
$\mathcal{C}^{\sharp}(B_p,B_n)\subseteq \mathcal{C}(B_p,B_n)$.
For a $t$-chain coloring $\chi$ of $B_n$, say that $\chi$ is \emph{weak-valid} if every $C\in \mathcal{C}(B_p,B_n)$ contains at least $q$ distinct colors on its $t$-chains, and \emph{strong-valid} if the same holds for every $C\in \mathcal{C}^{\sharp}(B_p,B_n)$.
Let $\mathcal{X}_{\mathrm{weak}}$ and $\mathcal{X}_{\mathrm{strong}}$ denote the sets of weak-valid and strong-valid colorings, respectively.
By the inclusion above, any weak-valid coloring is automatically strong-valid, hence
$\mathcal{X}_{\mathrm{weak}}\subseteq \mathcal{X}_{\mathrm{strong}}$.
By definitions of $f_t(n,p,q)$ and $f_t^{\sharp}(n,p,q)$, we have $f_t(n,p,q)\ge f_t^{\sharp}(n,p,q)$.
\end{proof}

We next give monotonicity in $p$.
\begin{proposition}\label{prop-ft-p-ineq}
Let $t, p, q, n$ be positive integers satisfying $1 \leq t < p+1$, $1 \leq q \leq h_{p-1}(t)$, and $n \geq p$. Then,
\[ f_t(n, p, q) \leq f_t(n, p-1, q) \quad \text{and} \quad f_t^\sharp(n, p, q) \leq f_t^\sharp(n, p-1, q). \]
\end{proposition}
\begin{proof}
Let $\alpha:B_{p-1}\rightarrow B_p$ be the natural inclusion.
Fix a $(p\!-\!1,q,t)$-valid coloring $\chi$ of the $t$-chains of $B_n$.
For any weak embedding $f:B_p\to B_n$, the composition $f\circ\alpha:B_{p-1}\to B_n$ is a weak embedding, and
$C'=(f\circ\alpha)(B_{p-1})\subseteq C=f(B_p)$.
Thus the set of colors appearing on $t$-chains of $C$ contains that of $C'$, so
$|\chi(A_t(C))|\ge |\chi(A_t(C'))|\ge q$.
Hence $\chi$ is $(p,q,t)$-valid, and therefore $f_t(n,p,q)\le f_t(n,p-1,q)$.
The same argument applies $f_t^\sharp(n,p,q)\le f_t^\sharp(n,p-1,q)$.
\end{proof}

We state monotonicity in $q$.
\begin{proposition}\label{prop-ft-q-ineq}
Let $t, p, q, n$ be positive integers satisfying
$1 \leq t < p+1$, $2 \leq q \leq h_p(t)$, and
$n \geq p$. Then
\[
f_t^\sharp(n, p, q-1) \leq f_t^\sharp(n, p, q).
\]
\end{proposition}

\begin{proof}
Since any strong $(p,q,t)$-coloring is also a strong $(p,q-1,t)$-coloring, a coloring that is valid for $(p,q,t)$ is also valid for $(p,q-1,t)$, and hence
$f_t^\sharp(n,p,q-1) \le f_t^\sharp(n,p,q)$.
\end{proof}

We now establish a fundamental monotonicity property of \(f_t^\sharp(n,p,2)\), which is critical for our inductive argument on Ramsey number bounds.
\begin{proposition}[Monotonicity of \(f_t^\sharp\)]\label{prop-ftsharp-n-ineq}
Let \(1 \leq t \leq p\) and \(p \leq n_1 \leq n_2\). Then
\[
f_t^\sharp(n_1, p, 2) \leq f_t^\sharp(n_2, p, 2).
\]
\end{proposition}
\begin{proof}
Let \(k = f_t^\sharp(n_2, p, 2)\). By Definition~\ref{def:ftnpq}, there exists a strong \((p,2,t)\)-coloring \(\chi\) of \(B_{n_2}\) with exactly \(k\) colors. Let \(\iota: B_{n_1} \hookrightarrow B_{n_2}\) denote the canonical inclusion (induced by \([n_1] \subseteq [n_2]\)), which maps induced copies of \(B_p\) in \(B_{n_1}\) to induced copies of \(B_p\) in \(B_{n_2}\), and \(t\)-chains in \(B_{n_1}\) to \(t\)-chains in \(B_{n_2}\).
The restriction \(\chi \circ \iota: \mathcal{T}_{n_1}(t) \to [k]\) is a strong \((p,2,t)\)-coloring of \(B_{n_1}\): every induced \(B_p\) in \(B_{n_1}\) maps to an induced \(B_p\) in \(B_{n_2}\), so its \(t\)-chains receive at least 2 distinct colors under \(\chi \circ \iota\) (by validity of \(\chi\)).
By minimality of \(f_t^\sharp(n_1, p, 2)\) (Definition~\ref{def:ftnpq}), the minimal number of colors for \(B_{n_1}\) cannot exceed \(k\). Thus \(f_t^\sharp(n_1, p, 2) \leq k = f_t^\sharp(n_2, p, 2)\).
\end{proof}

\section{Existence of strong Boolean Ramsey numbers}\label{sec6}

We begin by making precise the ordered structures that appear in the proof.

For posets $P$ and $Q$, we write $P\subseteq_{\ind} Q$ if $P$ is isomorphic to an
induced subposet of $Q$, that is, if there exists an injective map
$f:P\to Q$ such that for all $x,y\in P$, one has
$x\le_P y$ if and only if $f(x)\le_Q f(y)$.

Let $(P,\le_P)$ be a finite poset. A \emph{linear extension} of $(P,\le_P)$ is a
linear order $\prec_P$ on $P$ such that $x\le_P y$ implies $x\prec_P y$ for all
$x,y\in P$. A \emph{finite poset with a linear extension} is a triple
$\mathbf P=(P,\le_P,\prec_P)$, where $(P,\le_P)$ is a finite poset and
$\prec_P$ is a linear extension of $\le_P$.

Let $\mathbf A=(A,\le_A,\prec_A)$ and $\mathbf B=(B,\le_B,\prec_B)$ be finite
posets with linear extensions. A map $f:A\to B$ is called a \emph{copy} of
$\mathbf A$ in $\mathbf B$ if $f$ is injective and for all $x,y\in A$,
$x\le_A y$ if and only if $f(x)\le_B f(y)$, and
$x\prec_A y$ if and only if $f(x)\prec_B f(y)$.

\begin{theorem}[Ramsey theorem for finite posets with a linear extension {\upshape\cite[Theorem~6]{ArmanRodl17}}]\label{th:ordered-poset-ramsey}
Let $\mathbf A=(A,\le_A,\prec_A)$ and $\mathbf B=(B,\le_B,\prec_B)$ be finite
posets with linear extensions, and let $k\ge 2$. Then there exists a finite
poset with a linear extension $\mathbf C=(C,\le_C,\prec_C)$ such that for every
$k$-coloring of the set of copies of $\mathbf A$ in $\mathbf C$, there exists a
copy $\widetilde{\mathbf B}$ of $\mathbf B$ in $\mathbf C$ such that all copies
of $\mathbf A$ contained in $\widetilde{\mathbf B}$ receive the same color.
\end{theorem}
To transfer this statement to the Boolean lattice setting, we next record that every finite poset
admits a strong embedding into a Boolean lattice.
\begin{proposition}\label{prop:strong-embed-boolean}
Every finite poset $Q$ strongly embeds into a Boolean lattice. In fact,
$Q\subseteq_{\ind} B_{|Q|}$.
\end{proposition}

\begin{proof}
Let the ground set of $Q$ be $V(Q)$. For each $x\in Q$, define
$\phi(x)=\{y\in V(Q): y\le_Q x\}\subseteq V(Q)$. If $x\le_Q z$, then clearly
$\phi(x)\subseteq \phi(z)$. Conversely, if $\phi(x)\subseteq \phi(z)$, then
$x\in \phi(x)\subseteq \phi(z)$, hence $x\le_Q z$. Therefore
$x\le_Q z$ if and only if $\phi(x)\subseteq \phi(z)$, so $\phi$ is a strong
embedding of $Q$ into $B_{|Q|}$.
\end{proof}
The next two lemmas explain how the ordered setting above interacts with the
unordered poset language used in the definition of strong Boolean Ramsey numbers.

First, we identify the copies of an ordered chain with the ordinary chains in
the underlying poset.

\begin{lemma}\label{lem:copies-of-chain}
Let $\mathbf P=(P,\le_P,\prec_P)$ be a finite poset with a linear extension, and let
$\mathbf C_t$ denote the $t$-chain equipped with its unique linear extension.
Then the copies of $\mathbf C_t$ in $\mathbf P$ are exactly the $t$-chains of the
underlying poset $P$.
\end{lemma}

\begin{proof}
A copy of $\mathbf C_t$ in $\mathbf P$ is an isomorphic substructure of $\mathbf P$
with underlying set of size $t$. Since the underlying poset of $\mathbf C_t$ is a chain,
its image must be a subset of $P$ of size $t$ that is totally ordered by $\le_P$,
that is, a $t$-chain of $P$.

Conversely, let $X\subseteq P$ be a $t$-chain. Then the restriction of $\le_P$ to $X$
is a linear order. Since $\prec_P$ is a linear extension of $\le_P$, the restriction
of $\prec_P$ to $X$ extends that linear order. But on a finite set there is only one
linear order extending a given linear order, so the restrictions of $\le_P$ and
$\prec_P$ to $X$ coincide. Hence the induced ordered structure on $X$ is isomorphic
to $\mathbf C_t$, and therefore $X$ is a copy of $\mathbf C_t$ in $\mathbf P$.
\end{proof}
Next, we note that an ordered copy immediately yields an induced copy after the
linear extensions are forgotten.
\begin{lemma}\label{lem:ordered-copy-gives-induced}
Let $\mathbf Q=(Q,\le_Q,\prec_Q)$ and $\mathbf P=(P,\le_P,\prec_P)$ be finite
posets with linear extensions. If $\widetilde{\mathbf Q}$ is a copy of
$\mathbf Q$ in $\mathbf P$, then after forgetting the linear extensions,
$\widetilde{\mathbf Q}$ is an induced copy of the poset $Q$ in $P$.
\end{lemma}

\begin{proof}
Since $\widetilde{\mathbf Q}$ is a copy of $\mathbf Q$ in $\mathbf P$, there exists
a bijection $f:Q\to \widetilde Q\subseteq P$ such that for all $x,y\in Q$,
$x\le_Q y$ if and only if $f(x)\le_P f(y)$, and
$x\prec_Q y$ if and only if $f(x)\prec_P f(y)$. In particular,
$x\le_Q y$ if and only if $f(x)\le_P f(y)$. Therefore $f$ is a strong embedding of
the underlying poset $Q$ into the underlying poset $P$, and hence
$\widetilde Q=f(Q)$ is an induced copy of $Q$ in $P$.
\end{proof}
We are now in a position to prove the main result of this section.
\begin{theorem}\label{th:prob1-affirmative}
Let $k\ge 1$, let $t\ge 1$, and let $Q$ be a nonempty finite poset. Then
$\BRs_{k,t}(\mathcal B\mid Q)<\infty$. Consequently, Problem~\ref{prob1} has an
affirmative answer.
\end{theorem}

\begin{proof}
We consider two cases.

Suppose that $k=1$ or $Q$ has at most one $t$-chain.
By Proposition~\ref{prop:strong-embed-boolean}, there exists an integer $M$ such that
$Q\subseteq_{\ind} B_M$. If $k=1$, then every coloring of the $t$-chains of $B_M$
is monochromatic, so the induced copy of $Q$ inside $B_M$ is monochromatic. If $Q$
has at most one $t$-chain, then every induced copy of $Q$ is automatically monochromatic
under any $k$-coloring of the $t$-chains. In either case,
$\BRs_{k,t}(\mathcal B\mid Q)\le M<\infty$.

Suppose that $k\ge 2$ and $Q$ has at least two $t$-chains.
Fix a linear extension $\prec_Q$ of $Q$, and write
$\mathbf Q=(Q,\le_Q,\prec_Q)$. Let $\mathbf C_t$ denote the $t$-element chain
equipped with its unique linear extension.

Apply Theorem~\ref{th:ordered-poset-ramsey} with
$\mathbf A=\mathbf C_t$ and $\mathbf B=\mathbf Q$. Then there exists a finite poset
with a linear extension $\mathbf P=(P,\le_P,\prec_P)$ such that every $k$-coloring
of the copies of $\mathbf C_t$ in $\mathbf P$ contains a copy
$\widetilde{\mathbf Q}$ of $\mathbf Q$ in $\mathbf P$ all of whose copies of
$\mathbf C_t$ have the same color.

By Lemma~\ref{lem:copies-of-chain}, the copies of $\mathbf C_t$ in $\mathbf P$
are exactly the $t$-chains of the underlying poset $P$. Therefore every
$k$-coloring of the $t$-chains of $P$ contains a copy $\widetilde{\mathbf Q}$ of
$\mathbf Q$ such that all $t$-chains contained in $\widetilde{\mathbf Q}$ receive
the same color.

By Lemma~\ref{lem:ordered-copy-gives-induced}, after forgetting the linear extensions,
the copy $\widetilde{\mathbf Q}$ becomes an induced copy of the poset $Q$ in the
underlying poset $P$. Hence every $k$-coloring of the $t$-chains of $P$ contains a
monochromatic induced copy of $Q$.

Finally, by Proposition~\ref{prop:strong-embed-boolean}, there exists an integer $M$
and a strong embedding $\varphi:P\hookrightarrow B_M$. Let $\chi$ be any $k$-coloring
of the $t$-chains of $B_M$. Restrict $\chi$ to the $t$-chains of the induced copy
$\varphi(P)$. Transporting this coloring back to $P$ via $\varphi^{-1}$ gives a
$k$-coloring of the $t$-chains of $P$. By the previous paragraph, this coloring
contains a monochromatic induced copy $P_0\subseteq P$ of $Q$. Since $\varphi$ is a
strong embedding, $\varphi(P_0)$ is an induced copy of $Q$ in $B_M$, and all of its
$t$-chains receive the same color under $\chi$.

Therefore, every $k$-coloring of the $t$-chains of $B_M$ contains a monochromatic
induced copy of $Q$, so $\BRs_{k,t}(\mathcal B\mid Q)\le M<\infty$. This completes
the proof.
\end{proof}

The following corollary is immediate. 
\begin{corollary}\label{cor:existence-ft}
Let $n,p,t$ be integers with $p\ge 2$ and $1\le t\le p\le n$. Then
$f_t^\sharp(n,p,2)$
exists.
\end{corollary}
\begin{proof}
By Theorem~\ref{th:prob1-affirmative}, for every integer $k\ge 1$, 
$\BRs_{k,t}(\mathcal B\mid B_p)$
exists.
Consider the set
$S_n=\{k\ge 1:\ \BRs_{k,t}(\mathcal B\mid B_p)>n\}$.
We first show that $S_n$ is nonempty. Color each $t$-chain of $B_n$ with a distinct
color. Since every induced copy of $B_p$ contains at least two $t$-chains, no induced
copy of $B_p$ is monochromatic. Therefore, 
$\BRs_{h_n(t),t}(\mathcal B\mid B_p)>n$,
and hence $h_n(t)\in S_n$. Thus $S_n\neq\varnothing$.

Let $m=\min S_n$.
Since $m\in S_n$, we have $\BRs_{m,t}(\mathcal B\mid B_p)>n$. By the definition of the
strong Boolean Ramsey number, this means that there exists an $m$-coloring of the
$t$-chains of $B_n$ with no monochromatic induced copy of $B_p$. Equivalently, there
exists a strong $(p,2,t)$-coloring of $B_n$ with $m$ colors by \eqref{eq:ft-Ramsey-equivalence}.

Now let $1\le k<m$. Since $m$ is the least element of $S_n$, we have $k\notin S_n$, and hence
$\BRs_{k,t}(\mathcal B\mid B_p)\le n$.
Again by the definition of the strong Boolean Ramsey number, every $k$-coloring of the
$t$-chains of $B_n$ contains a monochromatic induced copy of $B_p$. Equivalently, no
strong $(p,2,t)$-coloring of $B_n$ exists with fewer than $m$ colors by \eqref{eq:ft-Ramsey-equivalence}.

Therefore $m$ is exactly the minimum number of colors in a strong $(p,2,t)$-coloring
of $B_n$. Hence
$f_t^\sharp(n,p,2)=m$.
In particular, $f_t^\sharp(n,p,2)$ exists.
\end{proof}

\section{Probabilistic upper bounds of $f_t^{\sharp}(n, p, q)$}
\label{sec3}

Let $A_1,\ldots,A_n$ be events in a probability space $\Omega$. We
say that the graph $\Gamma$ with vertex set $\{A_1,A_2,\ldots,A_n\}$ is a
\emph{dependency graph} if the edge set is defined as
$A_{i}$ and $A_j$ are adjacent if and only if $A_i$ and $A_{j}$ are dependent.

The systematic study of this function was initiated by Erdős and Gyárfás \cite{EG97} in 1997.

\begin{theorem}[\cite{Sp77}, Symmetric Lov\'asz Local Lemma]\label{th-LLL-SV}
Let \(A_1, A_2, \dots, A_m\) be events in a probability space. Let \(\Gamma\) be the dependency graph of \(\{A_i\}\) (edges between dependent events) with maximum degree \(d\). If \(\Pr[A_i] \leq \frac{1}{e(d + 1)}\) for all \(i\), then \(\Pr\left[\bigcap_{i=1}^m \overline{A_i}\right] > 0\).
\end{theorem}

Combining the three foundational tools established above, namely the exact count of \(t\)-chains in \(B_p\), the asymptotic bounds on the number of strong embeddings of \(B_p\) into \(B_n\), and the symmetric Lov\'asz Local Lemma, we can now derive the main upper bound for the function \(f_t^{\sharp}(n, p, q)\) as follows:
\newtheorem*{mainthm1}{\rm\bf Theorem~\ref{th-upper-Pr}}
\begin{mainthm1}[Restated]
Let \(p, q, t\) be positive integers with $1 \leq t< p+1$ and $1 \leq q \leq h_p(t)$ (where \(h_p(t)\) is defined in Proposition \ref{th-La2}). For all integers \(n \geq p\), there exists a constant \(c = c(p, q, t)\) such that
\begin{equation*}
f_t^{\sharp}(n, p, q) \leq c \cdot 2^{\frac{(n - p) \cdot \binom{p}{\lfloor p/2 \rfloor}(1+o(1))}{h_p(t) - q + 1}}.
\end{equation*}
\end{mainthm1}

\begin{proof}
If \(q=1\), then \(f_t^{\sharp}(n,p,1)=1\), so the conclusion is immediate. We therefore assume that \(q \geq 2\).
We use the probabilistic method combined with Theorem \ref{th-LLL-SV} to construct a valid \((p,q,t)\)-coloring and derive the upper bound for \(f_t^{\sharp}(n,p,q)\). Let
$$
K=c \cdot 2^{\frac{(n - p) \cdot \binom{p}{\lfloor p/2 \rfloor}(1+o(1))}{h_p(t) - q + 1}}.
$$
Recall that \(\mathcal{T}_n(t)\) denotes the set of all \(t\)-chains in \(B_n\).  Define a random coloring \(\chi: \mathcal{T}_n(t) \to \{1, 2, \dots, K\}\) where each \(T \in \mathcal{T}_n(t)\) is assigned a color uniformly at random. Color assignments are independent across distinct \(t\)-chains.

By Proposition \ref{th-La2}, every induced copy of $B_p$ contains exactly $h_p(t)$ $t$-chains.

For each induced copy \(S\subseteq B_n\) with \(S\cong_{\rm ind} B_p\), define the bad event
\[
A_S: |\chi(A_t(S))| \leq q-1.
\]
A valid \((p,q,t)\)-coloring requires no bad events.
By Theorem \ref{th-LLL-SV}, we verify two conditions:
\begin{itemize}
    \item[]  1. \(\Pr[A_S] \leq \frac{1}{e(d + 1)}\) for all induced copies \(S\subseteq B_n\);
\item[]  2. The maximum degree \(d\) of the dependency graph \(\Gamma\) is bounded.
\end{itemize}

Let \(h = h_p(t)\).
Each of the \(h\) \(t\)-chains has \(K\) color choices, so the total number of colorings is $K^h$.
For a fixed induced copy $S\subseteq B_n$ with $S\cong_{\rm ind} B_p$, a coloring of its $t$-chains is called \textit{bad} if it uses at most $q-1$ distinct colors, equivalently, the bad event $A_S$ occurs.
The number of such bad colorings is
$\sum_{s=1}^{q-1}\binom{K}{s} s^h.$

Using \(\binom{K}{s} \leq \frac{K^s}{s!}\) and \(\sum_{s=1}^{q-1} \frac{1}{s!} < e\), we have
\[
\sum_{s=1}^{q-1} \binom{K}{s} s^h \leq (q-1)^h K^{q-1} \sum_{s=1}^{q-1} \frac{1}{s!} \leq e (q-1)^h K^{q-1},
\]
and hence
the probability of \(A_S\) is
\begin{equation}\label{eq:prob-bad-event}
\Pr[A_S] \leq \frac{e (q-1)^h K^{q-1}}{K^h} = e (q-1)^h K^{-(h - q + 1)}.
\end{equation}

Two events \(A_S\) and \(A_{S'}\) are dependent if and only if the corresponding induced copies share at least one \(t\)-chain.
Let \(\mathcal{T}(S)=A_t(S)\). For each \(T \in \mathcal{T}(S)\), define \(\mathcal{F}(T)\) as the set of induced \(B_p\)-copies containing \(T\), excluding \(S\), so \(|\mathcal{F}(T)| \leq g(p,n)\).
By Theorem \ref{number_b_n},
\begin{equation}\label{eq:bound-n-t}
g(p,n)= 2^{(n - p) \cdot \binom{p}{\lfloor p/2 \rfloor}(1+o(1))}.
\end{equation}

The dependent induced copies for $A_S$ are contained in the union of the \(\mathcal{F}(T)\) over all $T\in \mathcal T(S)$. Hence
\[
\operatorname{deg}_{\Gamma}(A_S)
\leq \sum_{T \in \mathcal{T}(S)} |\mathcal{F}(T)|
\leq h\, g(p,n)
\leq h \cdot 2^{(n - p) \cdot \binom{p}{\lfloor p/2 \rfloor}(1+o(1))}.
\]
Therefore,
\begin{equation}\label{eq:bound-degree}
d \leq C \cdot 2^{(n - p) \cdot \binom{p}{\lfloor p/2 \rfloor}(1+o(1))},
\end{equation}
where \(C = h_p(t)\) is constant for fixed \(p,t\).

Since for large \(n\), \(d \geq 1\) and \(d + 1 \leq 2d\), substitute \(\Pr[A_S]\) from Equation \eqref{eq:prob-bad-event} and the upper bound of \(d\) from \eqref{eq:bound-degree}, we have
$$
\Pr[A_S] \cdot e(d + 1)\leq e \cdot (q-1)^h K^{-(h - q + 1)} \cdot e \cdot 2 C \cdot 2^{(n - p) \cdot \binom{p}{\lfloor p/2 \rfloor}(1+o(1))}.
$$
Let \(M = 2 C e^2 (q-1)^h\), a constant depending only on \(p,q,t\), and choose \(c = M^{\frac{1}{h - q + 1}}\). Then
\[\Pr[A_S] \cdot e(d + 1) \leq 1.\]

By Theorem \ref{th-LLL-SV}, we have
\(\Pr\left[\bigcap_{S\cong_{\rm ind} B_p} \overline{A_S}\right] > 0\).
Hence, there is a valid $(p,q,t)$-coloring in $B_{n}$ such that each induced $B_{p}$ copy has at least $q$ colors.
Since \(f_t^\sharp(n,p,q)\) is the minimum number of colors for a strong $(p,q,t)$-coloring, we have
\[f_t^{\sharp}(n,p,q) \leq K=  c \cdot 2^{\frac{(n - p) \cdot \binom{p}{\lfloor p/2 \rfloor}(1+o(1))}{h_p(t) - q + 1}}. \]
\end{proof}
\section[Lower bounds]{Lower bounds for \texorpdfstring{$f_t^\sharp(n,p,q)$}{ftsharp(n,p,q)}}\label{sec4}

We first estimate the extremal parameters \(\LaT_t^\sharp(n,B_m)\) and \(\LT_t^\sharp(n,B_m)\).

Let $h_{n,m}(t)$ denote the number of $t$-chains contained in the symmetric consecutive middle $m$ ranks of $B_n$.

\begin{theorem}\label{th-Lat}
Let $m\ge2$, $1\le t\le m$, and $n\ge m+1$. Then
\[
h_{n,m}(t)
\le
\LaT_t^\sharp(n,B_m)
\le
h_n(t)-\floor{\frac{n}{m+1}}\binom{m}{t-1},
\]
and
\[
\LT_t^\sharp(n,B_m)
\le
h_n(t)-\floor{\frac{n}{m+1}}.
\]
\end{theorem}

\begin{proof}
Let
\[
\mathcal M_{n,m}
=
\bigcup_{k=r}^{r+m-1}\mathcal S_k(B_n),
\qquad
r=\floor{\frac{n-m}{2}},
\]
where \(\mathcal S_k(B_n)=\{X\subseteq[n]:|X|=k\}\). Since \(\mathcal M_{n,m}\) has height \(m\), it contains no induced copy of \(B_m\), whose height is \(m+1\). Hence
$\LaT_t^\sharp(n,B_m)\ge h_{n,m}(t)$.

Set $s=\floor{\frac{n}{m+1}}$.
Partition the first \((m+1)s\) coordinates into disjoint blocks
\[
J_k=\{(k-1)(m+1)+1,\dots,k(m+1)\},
\qquad 1\le k\le s.
\]
Fix a distinguished element \(d_k\in J_k\) and write \(I_k=J_k\setminus\{d_k\}\). Define
\[
Q_k=\bigl\{\{d_k\}\cup A:A\subseteq I_k\bigr\}.
\]
Each \(Q_k\) is an induced copy of \(B_m\), and the families \(Q_1,\dots,Q_s\) are pairwise disjoint.

To prove the upper bound on \(\LaT_t^\sharp(n,B_m)\), let \(\mathcal F\subseteq B_n\) be strongly \(B_m\)-free. Then \(\mathcal F\cap Q_k\neq Q_k\) for every \(k\), so choose \(x_k\in Q_k\setminus \mathcal F\). Write \(x_k=\{d_k\}\cup A_k\) with \(|A_k|=r_k\). Fix an arbitrary linear ordering of \(I_k\). For every choice of subsets
\[
D\subseteq A_k,\qquad U\subseteq I_k\setminus A_k,\qquad |D|+|U|=t-1,
\]
we obtain a \(t\)-chain in \(Q_k\) through \(x_k\) by deleting the elements of \(D\) from \(x_k\) one by one in the chosen order and then adding the elements of \(U\) one by one in the chosen order. Distinct pairs \((D,U)\) yield distinct \(t\)-chains. Hence \(x_k\) belongs to at least
\[
\sum_{i=0}^{t-1}\binom{r_k}{i}\binom{m-r_k}{t-1-i}
=
\binom{m}{t-1}
\]
distinct \(t\)-chains of \(Q_k\), by Vandermonde's identity.
All these \(t\)-chains are absent from \(A_t(\mathcal F)\), and because the \(Q_k\) are pairwise disjoint, the omitted \(t\)-chains coming from different \(k\) are distinct. Therefore
\[
|A_t(\mathcal F)|\le h_n(t)-s\binom{m}{t-1}.
\]
Taking the maximum over all strongly \(B_m\)-free families \(\mathcal F\) yields
\[
\LaT_t^\sharp(n,B_m)\le h_n(t)-\floor{\frac{n}{m+1}}\binom{m}{t-1}.
\]

To prove the upper bound on \(\LT_t^\sharp(n,B_m)\), let \(\mathcal C_t\subseteq \mathcal T_n(t)\) be strongly \((t,B_m)\)-free. Then \(A_t(Q_k)\nsubseteq \mathcal C_t\) for every \(k\), so choose \(T_k\in A_t(Q_k)\setminus \mathcal C_t\). Since the copies \(Q_1,\dots,Q_s\) are pairwise disjoint, the sets \(A_t(Q_k)\) are pairwise disjoint, and thus the omitted chains \(T_1,\dots,T_s\) are distinct. Hence
\[
|\mathcal C_t|\le h_n(t)-s=h_n(t)-\floor{\frac{n}{m+1}}.
\]
Taking the maximum over all strongly \((t,B_m)\)-free sets \(\mathcal C_t\) gives the desired bound.
\end{proof}

The following double-counting lemma is the key ingredient for our second lower bound on \(f_t^\sharp\).

\begin{lemma}\label{lem-corrected-tchain-concentration}
Let $n\ge p\ge2$ and $1\le t\le p$. Let $\mathcal T\subseteq \mathcal T_n(t)$, and for each $T\in \mathcal T_n(t)$ let $C_T(p,n,t)$ denote the number of induced copies of $B_p$ in $B_n$ that contain $T$. Set
\[
C_{\min}(p,n,t;\mathcal T)=\min_{T\in \mathcal T} C_T(p,n,t).
\]
If
\[
|\mathcal T|
\ge
s\,\frac{g(p,n)}{C_{\min}(p,n,t;\mathcal T)}
\]
for some integer $s$ with $1\le s\le h_p(t)$, then there exists an induced copy $S\subseteq B_n$ of $B_p$ such that
$|A_t(S)\cap \mathcal T|\ge s$.
\end{lemma}

\begin{proof}
For each induced copy $S\subseteq B_n$ of $B_p$, let $t(S)=|A_t(S)\cap \mathcal T|$.

Count incidence pairs $(T,S)$ with $T\in\mathcal T$, \(S\cong_{\ind} B_p\), and \(T\in A_t(S)\). On the one hand,
\[
\sum_{S\cong_{\ind} B_p} t(S)
=
\sum_{T\in\mathcal T} C_T(p,n,t)
\ge
|\mathcal T|\,C_{\min}(p,n,t;\mathcal T).
\]
On the other hand, there are exactly \(g(p,n)\) induced copies of \(B_p\) in \(B_n\). Hence the average value of \(t(S)\) over all such copies is at least
\[
\frac{|\mathcal T|\,C_{\min}(p,n,t;\mathcal T)}{g(p,n)}.
\]
If the displayed quantity is at least \(s\), then some copy \(S\) satisfies \(t(S)\ge s\).
\end{proof}

\begin{theorem}\label{th:ftsharp-lower-bound-ultimate-rigorous}
Let $n\ge p\ge2$, let $1\le t\le p$, and let $2\le q\le h_p(t)$. Then
\[
f_t^\sharp(n,p,q)
\ge
\max\left\{
\ceil{\frac{h_n(t)}{\LT_t^\sharp(n,B_p)}},
\ceil{\frac{h_n(t)\,C_{\min}(p,n,t;\mathcal T_n(t))}{g(p,n)\,h_p(t)}}
\right\}.
\]
\end{theorem}

\begin{proof}
Let \(c=f_t^\sharp(n,p,q)\) and let
$\mathcal T_n(t)=\mathcal T_1\sqcup \cdots \sqcup \mathcal T_c$
be the partition into color classes in a strong \((p,q,t)\)-coloring of \(B_n\). Since \(q\ge2\), no color class contains all the \(t\)-chains of an induced copy of \(B_p\). Equivalently, each \(\mathcal T_i\) is strongly \((t,B_p)\)-free. By the pigeonhole principle there exists \(i_0\) such that
$|\mathcal T_{i_0}|\ge \frac{h_n(t)}{c}$.
Since \(\mathcal T_{i_0}\) is strongly \((t,B_p)\)-free,
$|\mathcal T_{i_0}|\le \LT_t^\sharp(n,B_p)$,
and therefore
\[
c\ge \ceil{\frac{h_n(t)}{\LT_t^\sharp(n,B_p)}}.
\]

For the second bound, apply Lemma~\ref{lem-corrected-tchain-concentration} to \(\mathcal T=\mathcal T_{i_0}\) with \(s=h_p(t)\). If
\[
|\mathcal T_{i_0}|
\ge
h_p(t)\,\frac{g(p,n)}{C_{\min}(p,n,t;\mathcal T_{i_0})},
\]
then some induced copy of \(B_p\) would have all its \(t\)-chains in \(\mathcal T_{i_0}\), contradicting that \(\mathcal T_{i_0}\) is strongly \((t,B_p)\)-free. Hence
$|\mathcal T_{i_0}|\,C_{\min}(p,n,t;\mathcal T_{i_0})
<
g(p,n)\,h_p(t)$.

Since \(C_{\min}(p,n,t;\mathcal T_n(t))\le C_{\min}(p,n,t;\mathcal T_{i_0})\) and \(|\mathcal T_{i_0}|\ge h_n(t)/c\), we obtain
\[
\frac{h_n(t)}{c}\,C_{\min}(p,n,t;\mathcal T_n(t))
\le
g(p,n)\,h_p(t),
\]
which implies
\[
c\ge
\ceil{\frac{h_n(t)\,C_{\min}(p,n,t;\mathcal T_n(t))}{g(p,n)\,h_p(t)}}.
\]
Taking the maximum of the two bounds completes the proof.
\end{proof}
Combining Theorem~\ref{th:ftsharp-lower-bound-ultimate-rigorous} with Theorem~\ref{th-Lat}, we show Theorem \ref{th123}.
\newtheorem*{mainthm2}{\rm\bf Theorem~\ref{th123}}
\begin{mainthm2}[Restated]
Let $n\ge 2p+1$, let $1\le t\le p$, and let $2\le q\le h_p(t)$. Then
\[
f_t^\sharp(n,p,q)
\ge
\max\left\{
\ceil{\frac{h_n(t)}{h_n(t)-\floor{\frac{n}{p+1}}}},
\ceil{\frac{h_n(t)\,C_{\min}(p,n,t;\mathcal T_n(t))}{g(p,n)\,h_p(t)}}
\right\}.
\]
\end{mainthm2}

\section{Applications of the Erd\H{o}s-Gy\'{a}rf\'{a}s function}\label{sec5}

With the upper bound for $f_t^\sharp(n,p,q)$ established, we next give logarithmic lower bounds for strong Boolean Ramsey numbers.
For a finite poset $Q$, let $f_t^\sharp(n,Q,q)$ denote the minimum number of colors in a coloring of the $t$-chains of $B_n$ such that every induced copy of $Q$ in $B_n$ uses at least $q$ distinct colors.

\begin{proposition}\label{prop:general-log-lower}
Let $Q$ be a nonempty finite poset, let $m=|Q|$, and let $h(Q,t)\ge2$. Then there exists a constant $c_Q=c_Q(t)>0$ such that
\[
f_t^\sharp(n,Q,2)
\le
c_Q\,2^{\frac{mn}{h(Q,t)-1}}
\]
for all sufficiently large $n$. Consequently,
\[
\BRs_{k,t}(\mathcal B\mid Q)
\ge
\frac{h(Q,t)-1}{m}\bigl(\log_2 k-\log_2 c_Q\bigr)
\qquad (k\to\infty).
\]
\end{proposition}

\begin{proof}
Write $h=h(Q,t)$ and $m=|Q|$. Color each $t$-chain of $B_n$ independently and uniformly from $[K]$. For an induced copy $S\subseteq B_n$ of $Q$, let $A_S$ be the bad event that the $t$-chains of $S$ use at most one color. As in the proof of Theorem~\ref{th-upper-Pr}, we have
$\Pr[A_S]\le eK^{-(h-1)}$.

The number of induced copies of $Q$ in $B_n$ is at most $(2^n)^m=2^{mn}$, because each copy is determined by the images of the $m$ elements of $Q$. Hence the dependency degree of the bad-event graph is at most $h\,2^{mn}$. Choosing
\[
K=c_Q\,2^{\frac{mn}{h-1}}
\qquad\text{with}\qquad
c_Q=(2e^2h)^{1/(h-1)},
\]
the symmetric Lov\'asz local lemma implies that with positive probability no bad event occurs. Therefore
\[
f_t^\sharp(n,Q,2)
\le
c_Q\,2^{\frac{mn}{h(Q,t)-1}}
\]
for all sufficiently large $n$.

Now let
\[
n_k=\left\lfloor \frac{h(Q,t)-1}{m}\bigl(\log_2 k-\log_2 c_Q\bigr)\right\rfloor.
\]
For all sufficiently large $k$, the first part of the proof gives $f_t^\sharp(n_k,Q,2)\le k$. By the definition of the strong Boolean Ramsey number, this implies
\[
\BRs_{k,t}(\mathcal B\mid Q)>n_k,
\]
and hence
\[
\BRs_{k,t}(\mathcal B\mid Q)
\ge
\frac{h(Q,t)-1}{m}\bigl(\log_2 k-\log_2 c_Q\bigr).
\]
\end{proof}

For Boolean lattices themselves, Theorem~\ref{th-upper-Pr} yields a sharper logarithmic lower bound.

\begin{corollary}\label{cor-R-lower}
Let $p\ge2$ and $1\le t\le p$, and let $c=c(p,2,t)$ be the constant from Theorem~\ref{th-upper-Pr}. Then, as $k\to\infty$,
\[
\BRs_{k,t}(\mathcal B\mid B_p)
\ge
p+\frac{(h_p(t)-1)\bigl(\log_2 k-\log_2 c\bigr)}{\binom{p}{\floor{p/2}}(1+o(1))}.
\]
\end{corollary}

\begin{proof}
Set
\[
C_p=\binom{p}{\floor{p/2}}.
\]
By Theorem~\ref{th-upper-Pr}, for every $\varepsilon>0$ there exists $n_0(\varepsilon)$ such that for all $n\ge n_0(\varepsilon)$,
\[
f_t^\sharp(n,p,2)
\le
c\,2^{\frac{(1+\varepsilon)(n-p)C_p}{h_p(t)-1}}.
\]
Fix $\varepsilon>0$ and define
\[
n_{k,\varepsilon}=\left\lfloor p+\frac{(h_p(t)-1)\bigl(\log_2 k-\log_2 c\bigr)}{(1+\varepsilon)C_p}\right\rfloor.
\]
For all sufficiently large $k$, we have $n_{k,\varepsilon}\ge n_0(\varepsilon)$ and therefore $f_t^\sharp(n_{k,\varepsilon},p,2)\le k$. By \eqref{eq:ft-Ramsey-equivalence}, this implies
\[
\BRs_{k,t}(\mathcal B\mid B_p)>n_{k,\varepsilon}.
\]
Since $\varepsilon>0$ is arbitrary, we obtain
\[
\BRs_{k,t}(\mathcal B\mid B_p)
\ge
p+\frac{(h_p(t)-1)\bigl(\log_2 k-\log_2 c\bigr)}{C_p(1+o(1))},
\]
as claimed.
\end{proof}

We now rigorously compare Corollary~\ref{cor-R-lower} with the lower bound of Katona, Mao, Ozeki, and Wang from Corollary~\ref{cor-lower-Boolean-uniform}.

Fix integers $p\ge2$, $k\ge3$, and set
\[
C_p=\binom{p}{\floor{p/2}}.
\]
Let $\mathrm{R}_{1}(k,p,t)$ denote the lower bound from Corollary~\ref{cor-lower-Boolean-uniform}. Then
\[
\BRs_{k,t}(\mathcal B\mid B_p) > \mathrm{R}_{1}(k,p,t)=\min\{T_1(k),T_2\},
\]
where
\[
\begin{aligned}
T_1(k)&=p+\frac{h_p(t)+(h_p(t)-1)\log_2(k-1)-1}{4C_p},\\
T_2&=p+\frac{h_p(t)-1}{2C_p}.
\end{aligned}
\]
Let $\mathrm{R}_{2}(k,p,t)$ denote our lower bound from Corollary~\ref{cor-R-lower}, namely
\[
\mathrm{R}_{2}(k,p,t)
=
p+\frac{(h_p(t)-1)(\log_2 k-\log_2 c)}{C_p(1+o(1))}.
\]

\begin{lemma}\label{lem-DKB}
For every $k\ge3$,
\[
\min\{T_1(k),T_2\}=T_2.
\]
\end{lemma}

\begin{proof}
The equality $T_1(k)=T_2$ is equivalent to
\[
h_p(t)+(h_p(t)-1)\log_2(k-1)-1=2(h_p(t)-1),
\]
which reduces to $\log_2(k-1)=1$, that is, $k=3$. Since $T_1(k)$ is strictly increasing in $k$, it follows that $T_1(k)\ge T_2$ for every $k\ge3$.
\end{proof}

\begin{proposition}\label{pro-compa}
Our lower bounds improve upon Corollary~\ref{cor-lower-Boolean-uniform} in the following two senses:
\begin{enumerate}
    \item[(i)]
    \[
    \lim_{k\to\infty}\frac{\mathrm{R}_{2}(k,p,t)}{\mathrm{R}_{1}(k,p,t)}=+\infty,
    \]
    and hence there exists $k_0=k_0(p,t)$ such that $\mathrm{R}_{2}(k,p,t)>\mathrm{R}_{1}(k,p,t)$ for all $k>k_0$.
    \item[(ii)] Proposition~\ref{prop:general-log-lower} yields a logarithmic lower bound for every finite poset $Q$, whereas Corollary~\ref{cor-lower-Boolean-uniform} applies only to $Q=B_p$ and requires $t\ge2$.
\end{enumerate}
\end{proposition}

\begin{proof}
By Lemma~\ref{lem-DKB}, we have $\mathrm{R}_{1}(k,p,t)=T_2=O(1)$ as $k\to\infty$. On the other hand,
\[
\mathrm{R}_{2}(k,p,t)
=
\frac{h_p(t)-1}{C_p(1+o(1))}\log_2 k+O(1)
=
\Omega(\log k).
\]
Therefore
\[
\lim_{k\to\infty}\frac{\mathrm{R}_{2}(k,p,t)}{\mathrm{R}_{1}(k,p,t)}=+\infty,
\]
which proves part~(i). Part~(ii) follows directly from Proposition~\ref{prop:general-log-lower} and Corollary~\ref{cor-lower-Boolean-uniform}.
\end{proof}

\section{Future work}\label{sec7}

We conclude with several natural problems suggested by the present work.

\begin{problem}
Determine the exact value of $\LaT_t^\sharp(n,B_m)$ for large $n$, and develop analogous exact or asymptotic results for more general target posets.
\end{problem}

\begin{problem}
Clarify when the weak and strong Boolean lattice Erd\H{o}s-Gy\'{a}rf\'{a}s functions coincide, and construct explicit examples that exhibit a strict gap.
\end{problem}

\begin{problem}
Extend the present framework to other host poset families, such as grids, distributive lattices, or random posets.
\end{problem}

\section*{Acknowledgment}

The authors thank Kenta Ozeki for helpful comments and suggestions.

\end{document}